\documentclass[preprint,12pt]{elsarticle}

\usepackage{amssymb}
\usepackage{amsmath}
\usepackage{amsthm}

\usepackage{orcidlink}

\journal{Journal of Mathematical Analysis and Applications }
\usepackage{hyperref}

\usepackage{xcolor}

\usepackage{enumitem}

\usepackage[framemethod=TikZ]{mdframed}
\newmdenv[
    frametitle=TO-DO,
    frametitlefont=\scshape,
    backgroundcolor=lightgray, 
    roundcorner=6pt,
    leftmargin=30,
    rightmargin=30,
    linecolor=red,
    linewidth=1pt,
    skipabove=3pt,
    skipbelow=4pt,
    font=\sffamily\itshape
]{TODO}

\newcommand*{\cinv}[2]{\left[#1 , #2\right]}
\newcommand*{\ocinv}[2]{\left(#1 , #2\right]}
\newcommand*{\coinv}[2]{\left[#1 , #2\right)}

\newcommand*{\set}[2]{\left\{ #1\colon #2\right\}}

\newcommand*{\abs}[1]{\left\vert #1\right\vert}
\newcommand*{\norm}[1]{\left\Vert #1\right\Vert}
\newcommand*{\lipnorm}[1]{\left\Vert #1\right\Vert_{\mathrm{lip}}}

\newcommand*{\ball}[2]{B\left( #1, #2 \right)}

\newcommand*{\charfunction}[1]{\chi_{#1}}

\newcommand*{\bR}{\mathbb{R}}

\newcommand*{\bL}{\mathbb{L}}

\newcommand*{\Neig}[1]{\mathcal{U}\left(#1\right)}

\newcommand*{\diff}{\mathop{}\!\mathrm{d}}
\newcommand*{\frechet}[2]{\diff #1\left(#2\right)}

\begin{document}
\sloppy

\newcommand{\Lip}{\operatorname{Lip}}
\newcommand{\lip}{\operatorname{lip}}
\newcommand{\LLip}{\operatorname{\mathbb{L}ip}}
\newcommand{\Int}{\operatorname{Int}}

\newtheorem{theorem}{Theorem}[section]
\newtheorem{lemma}[theorem]{Lemma}
\newtheorem{proposition}[theorem]{Proposition}
\newtheorem{corollary}[theorem]{Corollary}
\newtheorem{remark}[theorem]{Remark}

\theoremstyle{definition}
\newtheorem{definition}{Definition}
\newtheorem{example}{Example}
\newtheorem{problem}{Problem}

\begin{frontmatter}


 
\title{Classification of Lipschitz derivatives in terms of semicontinuity and the Baire limit functions \tnoteref{Grant}}
\tnotetext[Grant]{
    The research was supported by the University of Silesia in Katowice, Mathematics Department (Iterative Functional Equations
and Real Analysis program)
	%This research did not receive any specific grant from funding agencies in the public, commercial, or not-for-profit sectors.
	}

\author[us]{Oleksandr V. Maslyuchenko \orcidlink{0000-0002-1493-9399}\corref{CorrAuthor}} %% Author name
\ead{ovmasl@gmail.com}
\cortext[CorrAuthor]{Corresponding author}
\author[us]{Ziemowit M. Wójcicki \orcidlink{0009-0005-6433-0708}}
\ead{ziemo1@onet.eu}
\affiliation[us]
{organization={Institute of Mathematics, University of Silesia in Katowice},
%Department and Organization
addressline={Bankowa 12}, 
            city={Katowice},
            postcode={40-007}, 
            %state={},
            country={Poland}}
%\affiliation[chnu]
%{organization={Department of Mathematics and Informatics, Yuriy Fedkovych Chernivtsi National University},
%	%Department and Organization
%	addressline={Kotsiubynskoho  2}, 
%	city={Chernivtsi},
%	postcode={58012}, 
%	%state={},
%	country={Ukraine}}
%% Abstract
\begin{abstract}
    We introduce the generalized notion of semicontinuity of
    a function defined on a topological space and derive the
    useful classification of the so-called Lipschitz derivatives
    of functions defined on a metric space.
    Secondly, we investigate some connections of the Lipschitz
    derivatives defined on normed spaces to the Fr\'{e}chet
    derivative and relations between little, big and local
    Lipschitz derivatives (denoted by $\lip f$, $\Lip f$ and $\LLip f$
    respectively) in terms of Baire limit functions.
    In particular, we prove that $\lip f$ is $\mathcal{F}_{\sigma}$-lower,
    $\Lip f$ is $\mathcal{F}_{\sigma}$-upper, $\LLip f$ is
    upper semicontinuous. Moreover, for a function $f$ defined on an open
    or convex subset of a normed space, the upper Baire limit function of functions
    $\lip f$ and $\Lip f$ are equal to $\LLip f$.
\end{abstract}



%% Keywords

\end{frontmatter}



\section{Introduction}

The Lipschitz derivatives are useful tools for investigation of different notions of differentiability.
For example, the big Lipschitz derivative $\Lip f$ of a given function
$f$ often occurs in theorems of Rademacher--Stepanov type (see for example \cite{Ra, St, MalZa, EvGa}).
Statements of this type usually involves the set $L(f)=\set{x\in\bR}{\Lip f(x)<\infty}$.
The local Lipschitz derivative $\LLip f$ together with $\Lip f$ characterize the local and pointwise Lipschitzness 
of functions defined on a metric space \cite{HerMas}.
The little Lipschitz derivative was introduced by Cheeger in \cite{Cheg} and together with
$\Lip f$ play important role in the research of the first order differential calculus in metric spaces.
The crucial fact is the purely metric character of the definitions of Lipschitz derivatives.
The above considerations lead to the natural question of characterizing the sets 
$\ell(f)$, $L(f)$ and $\bL(f)$ for functions defined on metric spaces.
In recent years, this problem has been investigated in many articles, such as \cite{RmZu, BuHaRmZu, Ha}
In particular, in \cite{BuHaRmZu} it was shown that $\ell(f)$ is a $G_{\delta\sigma}$-set
and $L(f)$ is an $F_{\sigma}$-set for a function $f\colon\bR\to\bR$.
In our approach, we introduce generalized notions of semicontinuity of a function
and classify the Lipschitz derivatives with help of these properties. This allows
us to easily derive the Borel type of sets $\ell(f)$, $L(f)$, $\bL(f)$ of
a function $f$ acting between arbitrary metric spaces.

Another interesting question to consider is whether for a given
triplet of functions $(u, v, w)$, there exists a continuous
function $f$ such that $\lip f=u$, $\Lip f=v$ and $\LLip f=w$.
Some advances in this direction were made in \cite{BuHaMaVe2020} and \cite{BuHaMaVe2021}.
Our characterization of Lipschitz derivatives in terms of generalized
semicontinuity constrains the possible choice of functions $(u,v,w)$.
Moreover, we obtain another necessary criterion for a triple $(u,v,w)$ 
defined on locally convex subset of a normed space:
the upper Baire limits functions of $u$ and $v$ are equal to $w$.

\section{Lipschitz derivatives}

Let $X$ be a metric space, $a\in X$ and $\varepsilon>0$. We always denote the metric on $X$ by $|\,\cdot\,-\,\cdot\,|_X$ and
\begin{align*}
	B(a,\varepsilon)&=B_X(a,\varepsilon)=\big\{x\in X\colon|x-a|_X<\varepsilon\big\},\\
	B[a,\varepsilon]&=B_X[a,\varepsilon]=\big\{x\in X\colon|x-a|_X\le\varepsilon\big\}.
\end{align*}

\begin{definition}
	Let $X$ and $Y$ be metric spaces, $f\colon X\to Y$ be a function, $x\in X$. Denote
	\begin{itemize}
		\item $\|f\|_{\lip}=\sup\limits_{u\ne v\in X}\frac1{|u-v|_X}\big|f(u)-f(v)\big|_Y$,
		\item $\LLip f(x)=\limsup\limits_{(u,v)\to (x,x)}\frac1{|u-v|_X}\big|f(u)-f(v)\big|_Y$,
		\item $\Lip f(x)=\limsup\limits_{u\to x}\frac1{|u-x|_X}\big|f(u)-f(x)\big|_Y$,
		\item $\lip f(x)=\liminf\limits_{r\to0^+}\sup\limits_{u\in B(x,r)}\frac1r\big|f(u)-f(x)\big|_Y$;
	\end{itemize}
The number $\|f\|_{\lip}$ is \emph{Lipschitz constant of $f$}.  The functions $\LLip f$,  $\Lip f$  and $\lip f$ are called the \emph{local, big and little Lipschitz derivative} respectively.
\end{definition}
We denote by $X^d$ the set of all non-isolated points of $X$.
Throughout the paper, we assume that $\sup\varnothing=0$. 
As a consequence of this assumption we have 
$\LLip f(x)=\Lip f(x)=\lip f(x)=0$ for any $x\in X\setminus X^d$. 

Obviously, if $Y$ is a normed space then
$\|\cdot\|_{\lip}$ is an extended seminorm
on $Y^X$ in the sense \cite{SaTaGa}. Moreover,
$\|\cdot\|_{\lip}$ is a norm on the space $\Lip_a(X,Y)$ of all
Lipschitz functions $f\colon {X\to Y}$ vanishing at some fixed point $a\in X$.

We introduce some auxiliary notations:
	\begin{itemize}
		\item $\LLip^rf(x)=\big\|f|_{B(x,r)}\big\|_{\lip}=\sup\limits_{u\ne v\in B(x,r)}\frac1{|u-v|_X}\big|f(u)-f(v)\big|_Y$
		\item $\Lip^r f(x)=\sup\limits_{u\in B(x,r)}\frac1r\big|f(u)-f(x)\big|_Y, \Lip^r_+f(x)=\sup\limits_{u\in B[x,r]}\tfrac1r\big|f(u)-f(x)\big|_Y$
		\item $\Lip_r f(x)=\sup\limits_{0<\varrho<r}\Lip^\varrho f(x),\ \ \ \ \ \lip_r f(x)=\inf\limits_{0<\varrho<r}\Lip^\varrho f(x)$;
	\end{itemize}
Therefore, the definitions of the Lipschitz derivatives might be rewritten as follows.
\begin{align}
	\label{eqn:ILip_definition}
        \LLip f(x)&=\inf\limits_{r>0}\LLip^rf(x)\\
        \label{eqn:lip_definition}
        \lip f(x)&=\liminf\limits_{r\to0^+}\Lip^r f(x)
\end{align}

Some authors (see, for example, \cite{Ha,BuHaMaVe2021,BuHaRmZu})
define  $\Lip f$  and $\lip f$ using the function
$\Lip^r_+f$ instead of $\Lip^rf$. 
In the case where $X$ is a normed space, we have $B[x,r]=\overline{B(x,r)}$. 
Therefore, $\Lip^rf(x)=\Lip^r_+f(x)$ for any continuous function $f$.
But the previous equality does not hold for the discrete metric on $X$,
nonconstant $f$ and $r=1$. However, we have the following.

\begin{proposition}
    Let $X$ and $Y$ be metric spaces and $f\colon X\to Y$ be a function.
    Then, for any non-isolated point $x\in X$, the following equalities hold \[
        \Lip f(x)=\limsup_{r\to 0^{+}}\Lip^rf(x)=\limsup_{r\to 0^{+}}\Lip_{+}^rf(x).
    \]
\end{proposition}
\begin{proof}
    Denote 
    \begin{align*}
        \alpha(r) &= \sup_{0<\rho<r}\Lip^{\rho} f(x), \\
        \beta(r)  &= \sup_{0<\rho<r}\Lip_{+}^{\rho} f(x), \\
        \gamma(r) &= \sup\left\{\frac{\vert f(u)-f(x)\vert_Y }{\vert u - x\vert_X}\colon u\in X, \; 0<\vert u-x\vert_X<r\right\}.
    \end{align*}
    Since $B(x,r)\subseteq B[x,r]$, we have $\alpha(r)\leq\beta(r)$. Next, we have
    \begin{align*}
        \beta(r) &= \sup_{0<\rho<r}\sup_{0<\vert u-x\vert_X\leq\rho}\tfrac{1}{\rho}\vert f(u)-f(x)\vert_Y \\
                 &\leq \sup_{0<\rho<r}\sup_{0<\vert u-x\vert_X\leq\rho}\tfrac{1}{\vert u-x\vert_X}\vert f(u)-f(x)\vert_Y = \gamma(r).
    \end{align*}
    On the other hand, we have
    \begin{align*}
        \gamma(r) &= \sup_{0<\rho<r}\sup_{\vert u-x\vert_X=\rho}\tfrac{1}{\rho}\vert f(u)-f(x)\vert_Y \\
                  &\leq \sup_{0<\rho<r}\sup_{0<\vert u-x\vert_X\leq\rho}\tfrac{1}{\rho}\vert f(u)-f(x)\vert_Y \\
                  &= \sup_{0<\rho<r}\Lip_{+}^{\rho}f(x) = \beta(r).
    \end{align*}
    We have shown, that $\alpha(r)\leq\beta(r)=\gamma(r)$ for $r>0$.
    It remains to show that $\alpha(r)=\beta(r)$, $r>0$.
    Let us assume that there exists $r_0>0$, such that $\alpha(r_0)<\beta(r_0)$.
    Denote $\varphi(r)=\Lip^r f(x)$ and $\psi(r)=\Lip_{+}^r f(x)$. Observe, that
    \begin{equation}\label{eq:some_nonsense}
        \rho\varphi(\rho)\leq\rho\psi(\rho)\leq r\varphi(r), \text{ for }  0<\rho<r.
    \end{equation}
    Since $\beta(r_0)=\sup\limits_{r<r_0}\psi(r)>\alpha(r_0)$, 
    there exists $r_1<r_0$ such that $\alpha(r_0)<\psi(r_1)$.
    Let $\varepsilon=\psi(r_1)-\alpha(r_0)>0$. Then, for any $r<r_0$,
    $\varphi(r)+\varepsilon\leq\alpha(r_0)+\varepsilon=\psi(r_1)$, so
    \begin{equation}\label{eq:some_more_nonsense}
        r_1\varphi(r)+r_1\varepsilon\leq r_1\psi(r_1)\leq r\varphi(r), \text{ for } r_1<r<r_0,
    \end{equation}
    where the second inequality follows from \eqref{eq:some_nonsense}.
    Note that
    \[
        \varphi(r)=\tfrac{1}{r}r\varphi(r)\leq\tfrac{1}{r_1}r_0\varphi(r_0), \text{ for } r_1<r<r_0,
    \]
    so, the function $\varphi$ is bounded on the interval $(r_1;r_0)$.
    However, by \eqref{eq:some_more_nonsense}, we have
    \[
        0 = \lim_{r\to r_{1}^{+}}(r-r_1)\varphi(r)\geq r_1\varepsilon >0,
    \]
    which is impossible. We have $\alpha(r)=\beta(r)=\gamma(r)$ for $r>0$.
    But $$\displaystyle\limsup_{r\to 0^{+}}\Lip^r f(x)=\inf_{r>0}\alpha(r),$$
    $$\displaystyle\limsup_{r\to 0^{+}}\Lip_{+}^r f(x)=\inf_{r>0}\beta(r)$$
    and $$\Lip f(x)=\inf\limits_{r>0}\gamma(r)$$ and the proof is finished.
\end{proof}

%On the other hand, for an arbitrary metric space $X$, $x\in X$ and $r>0$ we have that 
%\begin{align*}
%r\Lip^rf(x)&=\sup\limits_{u\in B(x,r)}\big|f(u)-f(x)\big|_Y\\
%&=\sup\limits_{0<\varrho<r}\sup\limits_{u\in B(x,\varrho)}\big|f(u)-f(x)\big|_Y\\
%&=\sup\limits_{0<\varrho<r}\varrho\Lip^\varrho f(x).
%\end{align*}
%Therefore,
%$$\varrho\Lip^\varrho f(x)\le
%\sup\limits_{u\in B[x,\varrho]}\big|f(u)-f(x)\big|_Y= \varrho\Lip^\varrho_+ f(x)\le r\Lip^r f(x)$$
%for any $0<\varrho<r$ and $\lim\limits_{\varrho\to r^-}\varrho\Lip^\varrho f(x)=r\Lip^r f(x)$. Consequently,
%$$\lim\limits_{\varrho\to r^-}\Lip^\varrho_+f(x)=\Lip^rf(x)\le \Lip^r_+f(x)$$
%for any $r>0$. Thus, we can easily prove that the following assertions for any metric spaces $X$ and $Y$, $f\colon X\to Y$ and $x\in X$: 
Note, that
\begin{align}
	\label{eqn:Lip_monotone}
	\Lip_r f(x)&\le \Lip_{r'}f(x)
	\ \ \text{and}\ \ 
	\lip_r f(x)\ge \lip_{r'}f(x) \text{ if }0<r<r',
\end{align}
So, the definitions and the previous proposition yield
\begin{align}
	\label{eqn:Lip_as_limsup}\Lip f(x)&=\inf\limits_{r>0}\Lip_rf(x)=\lim\limits_{r\to0^+}\Lip_rf(x),\\
	\label{eqn:lip_as_sup_lip_r}\lip f(x)&=\sup\limits_{r>0}\lip_rf(x)=\lim\limits_{r\to0^+}\lip_rf(x).
\end{align}
Therefore, it is easy to see that the following inequalities hold.
\begin{align}
	 \label{eqn:Lip_r_estimations}\lip_r f(x)&\le \Lip_r f(x)\le\LLip^r f(x)\text{ for any }r>0,\\
	\label{eqn:Lip_estimations}\lip f(x)&\le \Lip f(x)\le\LLip f(x).
\end{align}

\begin{definition}
	Let $X$ and $Y$ be metric spaces and $\gamma\ge 0$. A function ${f\colon X\to Y}$ is called
	\begin{itemize}
		\item  \emph{$\gamma$-Lipschitz}  if $\|f\|_{\lip}\le\gamma$;
		\item \emph{Lipschitz} if $\|f\|_{\lip}<\infty$;
		\item \emph{locally Lipschitz} if $\LLip f<\infty$. 
		\item \emph{pointwise Lipschitz} if $\Lip f<\infty$;
		\item \emph{weakly pointwise Lipschitz} if $\lip f<\infty$.
	\end{itemize} 
\end{definition}

Inequalities (\ref{eqn:Lip_estimations}) yield the next assertion.

\begin{proposition}\label{prop:properties_Lip_lip}
	Let $X$ and $Y$ be metric spaces, and $f\colon X\to Y$ be a function. Then $\mathbb{L}(f)\subseteq L(f)\subseteq\ell(f)$ and $\ell^\infty(f)\subseteq L^\infty(f)\subseteq\mathbb{L}^\infty(f)$.
\end{proposition}

\section{Connections of Lipschitz derivatives to classical notion of a derivative}

One might ask under what conditions the Lipschitz derivative (of any given type)
coincides with one of the ``traditional'' notions of the derivative
of a given function, provided that an appropriate derivative exists.
It is obvious that for a real differentiable function $f\colon\bR\to\bR$,
we have $\lip f(x)=\Lip f(x)=\abs{f(x)}$ at any $x\in\bR$.

In \cite{HerMas}, the following theorem was proved.
\begin{theorem}\label{thm:local_Lip_Frechet}
    Let $X$ and $Y$ be normed spaces, $G$ be an open subset of $X$,
    and $f\colon G\to Y$ have a locally bounded Gateaux derivative $f'$.
    Then, $\LLip f(x)=\limsup\limits_{u\to x}\norm{f'(x)}$ for $x\in X$
    and so, $f$ is locally Lipschitz. Moreover, if $f$ is $C^1$ function,
    then $\LLip f(x)=\norm{f'(x)}$, $x\in X$.
\end{theorem}
The following result was also stated in \cite{HerMas}, but the proof
contains a small blunder. Here, we provide the correct proof.

\begin{theorem}\label{thm:big_Lip_Frechet}
    Let $f\colon X\to Y$, where $X$ and $Y$ are normed space
    and assume that there exists the Fr\'{e}chet derivative
    $\frechet{f}{x_0}$ of $f$ at point $x_0\in X$.
    Then, $\Lip f(x_0)=\norm{\frechet{f}{x_0}}$.
\end{theorem}

\begin{proof}
    It is enough to consider the case $X\ne\{0\}$.
    Denote by $A=\frechet{x_0}{f}$ the Fr\'{e}chet derivative
    of $f$ at a point $x_0$. We have
    \begin{equation}\label{eq:A_is_Frechet}
        f(x)-f(x_0)=A(x-x_0)+\alpha(x), \text{ for } x\in X,
    \end{equation}
    where $\alpha$ is a function, such that 
    $\displaystyle\lim_{x\to x_0}\frac{\alpha(x)}{\norm{x-x_0}}=0$.
    By \eqref{eq:A_is_Frechet} we have
    \[
        \norm{f(x)-f(x_0)}\leq\norm{A}\norm{x-x_0}+\norm{\alpha(x)},
    \]
    hence
    \[
        \frac{\norm{f(x)-f(x_0)}}{\norm{x-x_0}}\leq \norm{A}+\norm{\frac{\alpha(x)}{\norm{x-x_0}}}.
    \]
    Thus,
    \begin{align*}
        \Lip f(x_0) &= \limsup_{x\to x_0}\frac{\norm{f(x)-f(x_0)}}{\norm{x-x_0}} \\
                    &\leq \lim_{x\to x_0}\left(\norm{A}+\norm{\frac{\alpha(x)}{\norm{x-x_0}}}\right) 
                    = \norm{A}.
    \end{align*}
    We want to prove the reverse inequality.
    Fix $\varepsilon>0$. Then, there exists $e\in X$ with $\norm{e}=1$ and
    such that
    \[
        \norm{Ae}\geq\norm{A}-\varepsilon.
    \]
    For any $t>0$, denote $x_t=x_0+te$. 
    Note that $t=\norm{x_t-x_0}$, and $x_t\to x_0$ when $t\to 0$.
    Therefore, by \eqref{eq:A_is_Frechet}
    \[
        f(x_t)-f(x_0)=A(x_t-x_0)+\alpha(x_t)=tAe+\alpha(x_t),
    \]
    so
    \begin{equation*}%\label{eq:Ae_rxt}
        \frac{f(x_t)-f(x_0)}{t}=Ae+\frac{\alpha(x_t)}{t},
    \end{equation*}
    and
    \[
        \frac{f(x_t)-f(x_0)}{\norm{x_t-x_0}}=Ae+\frac{\alpha(x_t)}{\norm{x_t-x_0}}.
    \]
    Hence, we have
    \begin{align*}
        \frac{\norm{f(x_t)-f(x_0)}}{\norm{x_t-x_0}} &\geq \norm{Ae}-\norm{\frac{\alpha(x_t)}{\norm{x_t-x_0}}} \\
        &\geq \norm{A}-\varepsilon-\norm{\frac{\alpha(x_t)}{\norm{x_t-x_0}}}\xrightarrow{t\to 0}\norm{A}-\varepsilon.
    \end{align*}
    Thus,
    \begin{align*}
        \Lip f(x_0) &=\limsup_{x\to x_0}\frac{\norm{f(x)-f(x_0)}}{\norm{x-x_0}} 
        = \inf_{r>0}\sup_{\norm{x-x_0}<r}\frac{\norm{f(x)-f(x_0)}}{\norm{x-x_0}} \\
        &\geq \inf_{r>0}\sup_{\norm{x_t-x_0}<r}\frac{\norm{f(x_t)-f(x_0)}}{\norm{x_t-x_0}}
        \geq \norm{A}-\varepsilon-\lim_{t\to 0}\norm{\frac{\alpha(x_t)}{\norm{x_t-x_0}}} \\
        &= \norm{A}-\varepsilon.
    \end{align*}
    Since the $\varepsilon$ was chosen arbitrarily, the proof is finished.
\end{proof}

\section{Semicontinuity with respect to a family of sets}

In this section we introduce some modification of semicontinuity, 
which will help us to classify the Lipschitz derivatives.

Let $X$ be a topological space and $\mathcal{A}$ be a family of subset of $X$. We denote
\begin{equation*}
\begin{split} 
	\mathcal{A}_{c}&=\Big\{X\setminus A\colon A\in\mathcal{A}\Big\},\\
	\mathcal{A}_\sigma&=\Big\{\bigcup_{n=1}^\infty A_n\colon A_n\in\mathcal{A}\text{ for all }n\in\mathbb{N}\Big\},\\
	\mathcal{A}_\delta&=\Big\{\bigcap_{n=1}^\infty A_n\colon A_n\in\mathcal{A}\text{ for all }n\in\mathbb{N}\Big\},
\end{split}
\end{equation*}
We will also combine this symbols. 
It is easy to check, for example, that $\mathcal{A}_{c\delta c}=\mathcal{A}_{\sigma}$, $\mathcal{A}_{\sigma c}=\mathcal{A}_{c \delta}$, $\mathcal{A}_{\delta c}=\mathcal{A}_{c\sigma}$ and so on.
If $\mathcal{T}$ denotes the topology of $X$,
then applying above notation to the family $\mathcal{A}=\mathcal{T}$,
the $\mathcal{T}_{\delta}$ is the familiar Borel class $\mathcal{G}_{\delta}$
of $G_{\delta}$-subsets of $X$. Complementary, the family $\mathcal{T}_{c\sigma}$
is the Borel class $\mathcal{F}_{\sigma}$ of $F_{\sigma}$-subsets of $X$.
 %  With such notation, if $\mathcal{T}$ is the topology of $X$,
 %  then the set $A\in X$ is a $G_{\delta}$-set if and only if $A$ is in $\mathcal{T}_{\delta}$.
 %  Complementary, the set $A$ is an $F_{\sigma}$-set if and only if it belongs
 %  to $\mathcal{T}_{c\sigma}$.
 %  We will use notation $\mathcal{G}_{\delta}$ for the family
 %  $\mathcal{T}_{\delta}$ and $\mathcal{F}_{\sigma}$ for the family $\mathcal{T}_{c\sigma}$,
 %  whenever we want to emphasize, that we work with the traditional borel classes.

We say that  $f\colon X\to\overline{\mathbb{R}}$ is an \textit{$\mathcal{A}$-upper 
($\mathcal{A}$-lower) semicontinuous function} if 
$f^{-1}\big([-\infty;\gamma)\big)\in\mathcal{A}$ 
(resp. $f^{-1}\big((\gamma,+\infty]\big)\in\mathcal{A}$) for any
$\gamma\in\mathbb{R}$. If $\mathcal{A}=\mathcal{T}$ is the topology of $X$, 
then we omit the symbol $\mathcal{A}$ in the previous definitions. 
For our purposes, the $\mathcal{F}_{\sigma}$-upper and lower
semicontinuous functions are particularly important.
% In the case where $\mathcal{A}$ is the family of all $F_\sigma$-subsets of $X$,
% we replace $\mathcal{A}$ with $F_\sigma$ in the previous definitions.  
% According to our notation, $f$ is $F_{\sigma}$-upper (lower) semicontinuous 
% if and only if it is $\mathcal{F}_{\sigma}$-upper (lower) semicontinuous.
% if and only if it is $\mathcal{T}_{c\sigma}$-upper (lower) semicontinuous.

\begin{proposition}\label{prop:duality_A-upper_semicontinuity}
Let $X$ be a topological space, $\mathcal{A}\subseteq 2^X$ and
$f\colon X\to\overline{\mathbb{R}}$ be an $\mathcal{A}$-upper
semicontinuous function. Then
\begin{itemize}
\item[$(i)$] $f^{-1}\big[[\gamma,+\infty]\big]\in\mathcal{A}_{c}$  for any $\gamma\in \mathbb{R}$;
\item[$(ii)$] $f^{-1}\big[[-\infty,+\infty)\big]\in\mathcal{A}_{\sigma}$
and, so, $f^{-1}\big[\{+\infty\}\big]\in\mathcal{A}_{\sigma c}$;
\item[$(iii)$] $f^{-1}\big[\{-\infty\}\big]\in\mathcal{A}_{\delta}$ and, so, $f^{-1}\big((-\infty,+\infty]\big)\in\mathcal{A}_{\delta c}$;
\item[$(iv)$] $f$ is $\mathcal{A}_{c\sigma}$-lower semicontinuous
\end{itemize} 
\end{proposition}
\begin{proof}
	$(i)$ For any $\gamma\in\mathbb{R}$ we have that $f^{-1}\big([-\infty,\gamma)\big)\in\mathcal{A}$ and then $$f^{-1}\big([\gamma,+\infty]\big)=X\setminus f^{-1}\big([-\infty,\gamma)\big)\in\mathcal{A}_c.$$
	
	$(ii)$ Since $f^{-1}\big[[-\infty,n)\big]\in\mathcal{A}$ for any $n\in\mathbb{N}$, 
    we conclude that 
    \[
        f^{-1}\big[\coinv{-\infty}{+\infty}\big]=\bigcup\limits_{n=1}^\infty f^{-1}\big[\coinv{-\infty}{n}\big]\in\mathcal{A}_{\sigma},
    \]
    and so, 
    $f^{-1}\big[\{+\infty\}\big]=X\setminus f^{-1}\big[\coinv{-\infty}{+\infty}\big] \in\mathcal{A}_{\sigma c}$.
	
	$(iii)$ Since $f^{-1}\big[\coinv{-\infty}{-n}\big]\in\mathcal{A}$ for any $n\in\mathbb{N}$, we have that 
    \[
        f^{-1}\big[\{-\infty\}\big]=\bigcap\limits_{n=1}^\infty f^{-1}\big[\coinv{-\infty}{-n}\big]\in\mathcal{A}_{\delta},
    \]
    and so,
	$f^{-1}\big[\ocinv{-\infty}{+\infty}\big]=X\setminus f^{-1}\big[\{-\infty\}\big] \in\mathcal{A}_{\delta c}$
	
	$(iv)$ Let $\gamma\in\mathbb{R}$ and $\gamma_n\downarrow \gamma$. 
    Since $f^{-1}\big[[\gamma_n;+\infty]\big]\in \mathcal{A}_c$ by $(i)$,
    we conclude that 
    \[
        f^{-1}\big[(\gamma;+\infty]\big]=\bigcup\limits_{n=1}^\infty f^{-1}\big[[\gamma_n;+\infty]\big]\in\mathcal{A}_{c\sigma}
    \]
	i.e. $f$ is $\mathcal{A}_{c\sigma}$-lower semicontinuous.
\end{proof}

\begin{proposition}\label{prop:sup_of_A-upper_semicontinuous}
	Let $X$ be a topological space, $\mathcal{A}\subseteq 2^X$, 
    $f_n\colon X\to\overline{\mathbb R}$ be an $\mathcal{A}$-upper
    semicontinuous function for any $n\in\mathbb{N}$ and 
    $f\colon X\to\overline{\mathbb{R}}$ be a function such that 
    $f(x)=\sup\limits_{n\in\mathbb{N}}f_n(x)$  for any $x\in X$. 
    Then $f$ is an $\mathcal{A}_{c\sigma}$-lower  semicontinuous function.
\end{proposition}
\begin{proof}
	 Consider $\gamma\in\mathbb{R}$. 
     By Proposition~\ref{prop:duality_A-upper_semicontinuity}$(iv)$ the functions $f_{n}$ are
     $\mathcal{A}_{c\sigma}$-lower semicontinuous.
     So, $f_n^{-1}\big[\ocinv{\gamma}{+\infty}\big]\in\mathcal{A}_{c\sigma}$ for any $n\in\mathbb{N}$.
     %Then $f_n^{-1}\big((\gamma+\infty]\big)\in\mathcal{A}_{c\sigma}$ by. 
     Consequently,
	\[
        f^{-1}\big[\ocinv{\gamma}{+\infty}\big]=\bigcup\limits_{n=1}^\infty f_n^{-1}\big[\ocinv{\gamma}{+\infty}\big]\in\mathcal{A}_{c\sigma}.
    \]
    Thus, $f$ is an $\mathcal{A}_{c\sigma}$\nobreakdash-lower  semicontinuous function.
\end{proof}

Observe that $f$ is an $\mathcal{A}$-upper semicontinuous function 
if and only if $-f$ is $\mathcal{A}$-lower semicontinuous. Therefore, 
using Proposition~\ref{prop:duality_A-upper_semicontinuity} and \ref{prop:sup_of_A-upper_semicontinuous}
with $g=-f$ we obtain the following two propositions.

\begin{proposition}\label{prop:duality_A-lower_semicontinuity}
	Let $X$ be a topological space, $\mathcal{A}\subseteq 2^X$ and $f\colon X\to\overline{\mathbb{R}}$ be an $\mathcal{A}$-lower semicontinuous function. Then
	\begin{itemize}
		\item[$(i)$] $f^{-1}\big[\cinv{-\infty}{\gamma}\big]\in\mathcal{A}_{c}$  for any $\gamma\in \mathbb{R}$;
		\item[$(ii)$] $f^{-1}\big[\ocinv{-\infty}{+\infty}\big]\in\mathcal{A}_{\sigma}$
		and, so, $f^{-1}\big[\{-\infty\}\big]\in\mathcal{A}_{\sigma c}$;
		\item[$(iii)$] $f^{-1}\big[\{+\infty\}\big]\in\mathcal{A}_{\delta}$ and, so, 
            $f^{-1}\big[\coinv{-\infty}{+\infty}\big]\in\mathcal{A}_{\delta c}$;
		\item[$(iv)$] $f$ is $\mathcal{A}_{c\sigma}$-upper semicontinuous.
	\end{itemize} 
\end{proposition}

\begin{proposition}\label{prop:inf_of_A-lower_semicontinuous}
	Let $X$ be a topological space, 
    $\mathcal{A}\subseteq 2^X$, $f_n\colon X\to\overline{\mathbb R}$ be 
    an $\mathcal{A}$-lower semicontinuous function for any $n\in\mathbb{N}$ 
    and $f\colon X\to\overline{\mathbb{R}}$ be a function such that 
    $f(x)=\inf\limits_{n\in\mathbb{N}}f_n(x)$  for any $x\in X$. 
    Then $f$ is an $\mathcal{A}_{c\sigma}$-upper semicontinuous function.
\end{proposition}

\section{Classification of the Lipschitz derivatives}

Now we pass to the investigation of the type of semicontinuity of Lipschitz derivatives of continuous functions. 
In \cite{BuHaRmZu} semicontinuity of Lipschitz derivatives of a continuous function
$f\colon \mathbb{R}\to\mathbb{R}$ was obtained from the continuity of $\Lip^rf$. 
But in the general situation, this function need not be continuous. 
Therefore, we prove semicontinuity of Lipschitz derivatives directly by the definitions.

\begin{lemma}\label{lem:lip_r_upperSC}
	Let $X$ and $Y$ be metric spaces, $f\colon X\to Y$ be a continuous function and $r>0$. 
    Then $\lip_r f\colon X\to\cinv{0}{+\infty}$ is an upper semicontinuous function.
\end{lemma}
\begin{proof}
	Let $x_0\in X$ and $\gamma>\lip_r f(x_0)$.  Then
	$$
	\inf_{\varrho<r}\Lip^\varrho f(x_0)=\lip_r f(x_0)<\gamma.
	$$
	So, there is positive $\varrho<r$ such that $\Lip^\varrho f(x_0)<\gamma$. Pick  $\gamma_1$ such that $\Lip^\varrho f(x_0)<\gamma_1<\gamma$. Then we choose $\varrho_1$ such that $\frac{\gamma_1}{\gamma}\varrho<\varrho_1<\varrho$. So, $\gamma \varrho_1>\gamma_1\varrho$.
	Therefore,
	$$
	\sup\limits_{u\in B(x_0,\varrho)}\big|f(u)-f(x_0)\big|_Y=\varrho\Lip^\varrho f(x_0)<\gamma_1 \varrho.
	$$
	Then 
	$$
	\big|f(u)-f(x_0)\big|_Y<\gamma_1 \varrho\ \ \ \text{ for any }\ \ \  u\in  B(x_0,\varrho).
	$$
	By the continuity of $f$ at $x_0$ there exists $\delta>0$ such that $\varrho_1+\delta<\varrho$ and $$\big|f(x)-f(x_0)\big|_X<\gamma \varrho_1-\gamma_1 \varrho\ \ \ \text{ for any }\ \ \ x\in U=B(x_0,\delta).$$
	Consider $x\in U$ and $u\in B(x,\varrho_1)$. 
	Then
	$$|u-x_0|_X\le |u-x|_X+|x-x_0|_X<\varrho_1+\delta<\varrho,$$
	and so, $u\in B(x_0,\varrho)$. Therefore,
	$$
	\big|f(u)-f(x)\big|_Y\le \big|f(u)-f(x_0)\big|_Y+\big|f(x_0)-f(x)\big|_Y<\gamma_1\varrho+(\gamma \varrho_1-\gamma_1 \varrho)=\gamma \varrho_1.
	$$
	Thus, $\frac1{\varrho_1}\big|f(u)-f(x)\big|_Y\le \gamma$ for any $u\in B(x,\varrho_1)$.
	Hence, $\Lip^{\varrho_1}f(x)\le\gamma$. But $0<\varrho_1<r$. Therefore, $\lip_r f(x)\le\gamma$ for any $x\in U$. Thus, $\lip_r f$ is upper semicontinuous at $x_0$.
\end{proof}
\begin{theorem}\label{thm:lip_Fs_lowerSC}
	Let $X$ and $Y$ be metric space, $f\colon X\to Y$ be a continuous function. 
    Then $\lip f\colon X\to\cinv{0}{+\infty}$ is a % $F_\sigma$-lower semicontinuous function.
    $\mathcal{F}_\sigma$-lower  semicontinuous function.
\end{theorem}
\begin{proof}
    By (\ref{eqn:Lip_monotone}) and (\ref{eqn:lip_as_sup_lip_r}) we conclude that
    $\lip f(x)=\sup\limits_{n\in\mathbb{N}}\lip_{\frac1n}f(x)$ for any $x\in X$. 
    By Lemma~\ref{lem:lip_r_upperSC}, the functions $\lip_{\frac1n}f$ are
    $\mathcal{T}$-upper semicontinuous, where $\mathcal{T}$ is the topology of $X$.
    Therefore, by Proposition~\ref{prop:sup_of_A-upper_semicontinuous} $\lip f$ is
    $\mathcal{T}_{c\sigma}$\nobreakdash-lower semicontinuous.
    % This means that $\lip f$ is $F_\sigma$-lower semicontinuous.
    This means that $\lip f$ is $\mathcal{F}_\sigma$-lower semicontinuous.
    %and Proposition~\ref{prop:inf_and_sup_of_semicontinuous}.
\end{proof}

\begin{lemma}\label{lem:Lip_r_lowerSC}
	Let $X$ and $Y$ be metric spaces, $f\colon X\to Y$ be a continuous function and $r>0$. 
    Then $\Lip_r f\colon X\to\cinv{0}{+\infty}$ is a lower semicontinuous function.
\end{lemma}
\begin{proof}
	Fix $r>0$. Let $x_0\in X$ and $\gamma<\Lip_r f(x_0)$.  Then
	$$
	\sup\limits_{\varrho<r}\Lip^\varrho f(x_0)=\Lip_r f(x_0)>\gamma.
	$$
	So, there is $\varrho\in(0;r)$ such that $\Lip^\varrho f(x_0)>\gamma$. 
	Pick $\gamma_1$ such that ${\gamma<\gamma_1<\Lip^\varrho f(x_0)}$. 
	Therefore,
	$$
	\sup\limits_{u\in B(x_0,\varrho)}\big|f(u)-f(x_0)\big|_Y=\varrho\Lip^\varrho f(x_0)>\gamma_1 \varrho.
	$$
	Thus, there is $u\in B(x_0,\varrho)$ with
	$$
	\big|f(u)-f(x_0)\big|_Y>\gamma_1 \varrho.
	$$
	Then we choose $\varrho_1$ such that $\varrho<\varrho_1<\min\big\{r,\frac{\gamma_1}{\gamma}\varrho\big\}$. Consequently, 
	$\gamma \varrho_1<\gamma_1\varrho$.
	By the continuity of $f$ at $x_0$ there exists $\delta>0$ such 
	that $\varrho+\delta<\varrho_1$ and 
	$$\big|f(x)-f(x_0)\big|_Y<\gamma_1\varrho-\gamma \varrho_1\ \ \ \text{ for any }\ \ \ x\in U:=B(x_0,\delta).$$ 
	Consider $x\in U$. Then 
	$$|u-x|_X\le |u-x_0|_X+|x_0-x|_X<\varrho+\delta<\varrho_1,$$
	and, so, $u\in B(x,\varrho_1)$. Consequently,
	$$
	\big|f(u)-f(x)\big|_Y\ge \big|f(u)-f(x_0)\big|_Y-\big|f(x)-f(x_0)\big|_Y>\gamma_1\varrho-(\gamma_1\varrho-\gamma \varrho_1)=\gamma \varrho_1.
	$$
	Hence, $\Lip^{\varrho_1}f(x)>\gamma$. But $0<\varrho_1<r$. Therefore, $\Lip_r f(x)>\gamma$ for any $x\in U$. Thus, $\Lip_r f$ is lower semicontinuous at $x_0$.
\end{proof}
\begin{theorem}\label{thm:Lip_Fs_upperSC}
	Let $X$ and $Y$ be metric spaces, $f\colon X\to Y$ be a continuous function. 
    Then $\Lip f\colon X\to\cinv{0}{+\infty}$ is a % $F_\sigma$-upper semicontinuous function.
    $\mathcal{F}_\sigma$-upper semicontinuous function.
\end{theorem}
\begin{proof}
	By (\ref{eqn:Lip_monotone}) and (\ref{eqn:Lip_as_limsup}) we conclude that
    $\Lip f(x)=\inf\limits_{n\in\mathbb{N}}\Lip_{\frac1n}f(x)$ for any $x\in X$. 
    By Lemma \ref{lem:Lip_r_lowerSC}, the functions $\Lip\limits_{\frac{1}{n}}f$
    are $\mathcal{T}$-lower semicontinuous where $\mathcal{T}$ is the topology of $X$.
    Therefore, by Proposition \ref{prop:inf_of_A-lower_semicontinuous}
    $\Lip f$ is $\mathcal{T}_{c\sigma}$\nobreakdash-upper semicontinuous.
    This means that $\Lip f$ is $\mathcal{F}_{\sigma}$-upper semicontinuous.
    % This means that $\Lip f$ is $F_{\sigma}$-upper semicontinuous.
    %Thus, the needed assertions is implied from Lemma~\ref{lem:Lip_r_lowerSC} and Proposition~\ref{prop:inf_and_sup_of_semicontinuous}.
\end{proof}
\begin{theorem}\label{thm:ILip_upperSC}
	Let $X$ and $Y$ be metric spaces, $f\colon X\to Y$ be a function. 
    Then $\LLip f\colon X\to\cinv{0}{+\infty}$ is an upper
    semicontinuous function.
\end{theorem}
\begin{proof}
	Fix $x_0\in X$ and $\gamma>\LLip f(x_0)$. 
	Since $\LLip f(x_0)=\inf\limits_{r>0}\LLip^rf(x_0)$, 
	there exists $r>0$ such that $\LLip^r f(x_0)<\gamma$. 
	Set $\varrho=\frac{r}{2}$ and consider $x\in B(x_0,\varrho)$. 
	Then $B(x,\varrho)\subseteq B(x_{0},r)$. Consequently, 
	\begin{align*}
		\LLip f(x)&\le\LLip^\varrho f(x)=\sup\limits_{u\ne v\in B(x,\varrho)}\tfrac{1}{|u-v|_X}\big|f(u)-f(v)\big|_Y\\
		&\le
		\sup\limits_{u\ne v\in B(x_0,r)}\tfrac{1}{|u-v|_X}\big|f(u)-f(v)\big|_Y
		=\LLip^r f(x_0)<\gamma
	\end{align*}
	and, hence, $\LLip f$ is upper semicontinuous.
\end{proof}
	Denote
\begin{itemize}
	\item $\mathbb{L}(f)=\big\{x\in X:\LLip f(x)<\infty \big\}$;
	\item $\mathbb{L}^\infty(f)=\big\{x\in X:\LLip f(x)=\infty \big\}=X\setminus \mathbb{L}(f)$;
	\item $L(f)=\big\{x\in X:\Lip f(x)<\infty \big\}$;
	\item $L^\infty(f)=\big\{x\in X:\Lip f(x)=\infty \big\}=X\setminus L(f)$;
	\item $\ell(f)=\big\{x\in X:\lip f(x)<\infty \big\}$;
	\item $\ell^\infty(f)=\big\{x\in X:\lip f(x)=\infty \big\}=X\setminus\ell(f)$;
\end{itemize}

Theorems~\ref{thm:lip_Fs_lowerSC}, \ref{thm:Lip_Fs_upperSC}, \ref{thm:ILip_upperSC},
and Propositions~\ref{prop:duality_A-upper_semicontinuity}, \ref{prop:duality_A-lower_semicontinuity}
yield the following assertions.

\begin{corollary}\label{prp:lL}
	Let $X$ and $Y$ be a metric space, $f\colon X\to Y$ be a continuous function, $x\in X$. Then
	\begin{enumerate}[label=$(\roman*)$]
		\item $\ell(f)$ is a $G_{\delta\sigma}$-set;
		\item $\ell^\infty(f)$ is an $F_{\sigma\delta}$-set;
		\item $L(f)$ is an $F_\sigma$-set;
		\item $L^\infty(f)$ is a $G_\delta$-set;
	\end{enumerate}
\end{corollary}
\begin{proof}
    $(i)$ We have
    \begin{align*}
        \ell(f) &= \set{x\in X}{\lip f(x) < \infty} \\
                &= (\lip f)^{-1}\big[\coinv{-\infty}{+\infty}] \\
                &= X\setminus(\lip f)^{-1}\big[\{+\infty\}]
    \end{align*}
    and, since $\lip f$ is 
    $\mathcal{F}_{\sigma}$-lower semicontinuous,
    by Proposition~\ref{prop:duality_A-lower_semicontinuity}$(iii)$
    $(\lip f)^{-1}\big[\{+\infty\}]\in\mathcal{F}_{\sigma\delta}$, so 
    $\ell(f)=X\setminus(\lip f)^{-1}\big[\{+\infty\}]\in\mathcal{G}_{\delta\sigma}$.
%   \begin{align*}
%       \ell(f) &= \set{x\in X}{\lip f(x) < \infty} \\
%               &= \set{x\in X}{\lip f(x)\leq k, \text{ for some } k\in\bN} \\
%               &= \bigcup_{k=1}^{\infty}(\lip f)^{-1}\big[\cinv{0}{k}\big].
%   \end{align*}
%   Denote
%   $G_k=(\lip f)^{-1}\big[\cinv{0}{k}\big]$ and observe, that
%   $G_k=X\setminus (\lip f)^{-1}\big[\ocinv{k}{+\infty}\big]$.
%   Since the set $(\lip f)^{-1}\big[\ocinv{k}{+\infty}\big]$ is of the $F_{\sigma}$ type by
%   $F_{\sigma}$-lower semicontinuity of $\lip f$, the set $G_k$ is a $G_{\delta}$ set for any $k$.
%   So $\ell(f)=\bigcup\limits_{k=1}^\infty G_k$ is a $G_{\delta\sigma}$-set.

    $(ii)$ It follows immediately from $(i)$. %, since $\ell^{\infty}(f)=X\setminus\ell(f)$.

    $(iii)$ It is easy to see, that $L(f)=\bigcup\limits_{k=1}^{\infty}(\Lip f)^{-1}[\coinv{0}{k}]$.
    Since $\Lip f$ is an $\mathcal{F}_{\sigma}$-upper semicontinuous function, each set
    $(\Lip f)^{-1}[\coinv{0}{k}]$ is of $F_{\sigma}$ type, hence $L(f)$ is an
    $F_{\sigma}$-set as a countable sum of $F_{\sigma}$-sets.

    $(iv)$ It follows from $(iii)$.
\end{proof}

\begin{corollary}\label{prp:IL}
	Let $X$ and $Y$ be a metric space, $f\colon X\to Y$ be a  function, $x\in X$. Then
	\begin{enumerate}[label=$(\roman*)$]
		\item $\mathbb{L}(f)$ is an open set;
		\item $\mathbb{L}^\infty(f)$ is a closed set.
	\end{enumerate}
\end{corollary}

\section{Characterization of Lipschitz functions on a convex subset of a normed space}

The following lemma was applied by Buczolich, Hanson, Maga and Vertesy
in certain investigations of Lipschitz derivatives of the real functions 
of real variable.
\begin{lemma}[{\cite[Lemma 2.2]{BuHaMaVe2021}}]\label{lem:BHMV}
    If $E\subseteq\bR$ and $f\colon\bR\to\bR$ such that
    $\lip f\leq\charfunction{E}$ then
    $\abs{f(a)-f(b)}\leq\mu(\cinv{a}{b}\cap E)$ for every
    $a,b\in\bR$ (where $a<b$) so $f$ is Lipschitz and hence 
    absolutely continuous. 
\end{lemma} 
In the above, $\mu$ denotes the Lebesgue measure.
We will state the following
%As an immediate consequence we get
 %  The following is the immediate consequence
 %  of Lemma \ref{lem:BHMV}. %lemma follows immediately from \cite[Lemma 2.2]{BuHaMaVe2021}.
\begin{corollary}\label{cor:BHMV}
	Let  $\gamma> 0$ and $f\colon\cinv{0}{1}\to\mathbb{R}$ 
    be a function such that $\lip f(x)\le \gamma $ for any 
    $x\in\cinv{0}{1}$. Then $f$ is $\gamma$-Lipschitz.
\end{corollary}
\begin{proof}
    Extend $f$ to $\tilde{f}\colon\mathbb{R}\to\mathbb{R}$ by 
    $\tilde{f}(x)=f(0)$ if $x<0$ and $\title{f}(x)=f(1)$ if $x>1$. 
    Let $g=\frac{1}{\gamma}\tilde{f}$ and $E=\cinv 01$. Then by 
    Lemma~\ref{lem:BHMV} we conclude that 
    $\frac{1}{\gamma}\abs{f(x)-f(y)}=\abs{g(x)-g(y)}\le\mu\big(\cinv{x}{y}\cap E\big)=\abs{x-y}$ 
    for any $x,y\in\cinv{0}{1}$.
\end{proof}
Next result will allow us to apply 
Lemma \ref{lem:BHMV} and Corollary \ref{cor:BHMV}
for functions defined on normed spaces.
\begin{lemma}\label{lem:lip_linear_composition}
    Let $A$ be a convex subset of the normed space $X$, 
    $f\colon X\to\bR$ be a function, and $a,b\in A$.
    Moreover, let $T\colon\cinv{0}{1}\to A$ be an 
    affine function given by 
    $T(u)=a+t(b-a)$ for $0\leq u\leq 1$
    and $g=f\circ T\colon\cinv{0}{1}\to\bR$.
    Then, 
    \begin{equation}\label{eq:linear_chain_rule}
        \lip g\leq\norm{b-a}\big((\lip f)\circ T\big).
    \end{equation}
\end{lemma}
As $\big((\lip f)\circ T\big)\norm{b-a}=\big((\lip f)\circ T)\cdot\norm{\frac{\diff T}{\diff u}}$,
the right side of the inequality
\eqref{eq:linear_chain_rule} 
is reminiscent of the ``chain rule''
for the usual derivative.
Nevertheless, the inequality can be strict.
To see that, take a function $f\colon\bR\to\bR$
defined as $f(x,y)=y$, $(x,y)\in\bR$ and consider
the usual distance on $\bR$. 
 %  We have, $\lip f(x,y)=1$.
    % $\lip f(x,y)=\norm{f'(x,y)}=\sqrt{\left(\frac{\partial f}{\partial x}(x,y)\right)^2+\left(\frac{\partial f}{\partial y}(x,y)\right)^2}=1$.
 %  Then, for $a=(0,0)$, $b=(1,0)$, $T(u)=a+(b-a)\lambda=(\lambda,0)$, 
 %  $u\in\cinv{0}{1}$ and
 %  $g=f\circ T$ we have $\lip g(\lambda_0)=\norm{b-a}=1$.
\begin{proof}[Proof of the Lemma \ref{lem:lip_linear_composition}]
    Fix $u_0\in\cinv{0}{1}$ and
    observe, that 
    \begin{equation}\label{eq:norm_of_linear_difference}
        \norm{T(u)-T(u_0)}=\norm{(u-u_0)(b-a)}=\abs{u-u_0}\norm{b-a}, \; 0\leq u\leq 1.
    \end{equation}
    Put $x_0=T(u_0)$.
    We have
    \begin{align}
        \lip g(u_0) &= \liminf_{r\to 0^{+}}\sup_{\abs{u-u_0}<r}\frac{\abs{g(u)-g(u_0)}}{r} \nonumber \\
        &= \liminf_{r\to 0^{+}}\sup_{0<\abs{u-u_0}<r}\frac{\abs{f(T(u))-f(T(u_0))}}{r} \nonumber \\
        &= \liminf_{r\to 0^{+}}\sup_{0<\abs{u-u_0}<r}
           \frac{\norm{T(u)-T(u_0)}}{r}\cdot\frac{\abs{f(T(u))-f(T(u_0))}}{\norm{T(u)-T(u_0)}} \nonumber \\
        &= \liminf_{r\to 0^{+}}\sup_{0<\abs{u-u_0}<r}\frac{\norm{b-a}\abs{u-u_0}}{r}
           \frac{\abs{f(T(u))-f(T(u_0))}}{\norm{T(u)-T(u_0)}}, \label{eq:lipg_form}
    \end{align}
    where the last equality follows from \eqref{eq:norm_of_linear_difference}.
   % \varphi here 
 %  Put $\varphi(r)=\sup\limits_{0<\abs{u-u_0}<r}\frac{\norm{b-a}\abs{u-u_0}}{r}$.
 %  We will show, that the functoin $\varphi\colon\oinv{0}{\infty}\to\coinv{0}{\infty}$ is (right side) continuous at $0$.
 %  Let $\varepsilon>0$. Then, for $\delta=\max\{2,\norm{b-a}^{-1}\varepsilon\}$, we have
 %  \[
 %      \frac{\norm{b-a}\abs{u-u_0}}{r}<\norm{b-a}\leq\frac{\varepsilon}{\delta}<\frac{\varepsilon}{2}<\varepsilon
 %  \]
 %  whenever $0<r<\delta$ and $\abs{u-u_0}<r$.
 %  That means that
 %  \[
 %      0\leq\varphi(r)=\sup_{\abs{u-u_0}<r}\frac{\norm{b-a}\abs{u-u_0}}{r}<\varepsilon, \text{ for } 0<r<\delta.
 %  \]
 %  As a consequence,
 %  \[
 %      \limsup_{r\to 0^{+}}\varphi(r)=\liminf_{r\to 0^{+}}\varphi(r)=\lim_{r\to 0^{+}}\varphi(r).
 %  \]
   % end of fun with varphi
    Note, that
    \begin{equation}\label{eq:normb_a}
        \sup_{\abs{u-u_0}<r}\frac{\norm{b-a}\abs{u-u_0}}{r}=\norm{b-a}.
    \end{equation}
    Thus, by \eqref{eq:lipg_form} and \eqref{eq:normb_a}
%   \[
%      \sup_{0<\abs{t-u_0}<r}\frac{\norm{b-a}\abs{t-u_0}}{r}
%      \frac{\abs{f(T(u))-f(T(u_0))}}{\norm{T(u)-T(u_0)}}
%      \leq\norm{b-a}\sup_{0<\abs{t-u_0}<r}\frac{\abs{f(T(u))-f(T(u_0))}}{\norm{T(u)-T(u_0)}}
%   \]
%    we can estimate right hand side of \eqref{eq:lipg_form}
    \begin{align*}
        \lip g(u_0) 
          %&= \liminf_{r\to 0^{+}}\sup_{0<\abs{u-u_0}<r}\frac{\norm{b-a}\abs{u-u_0}}{r}
          %\frac{\abs{f(T(u))-f(T(u_0))}}{\norm{T(u)-T(u_0)}} \\
           &\leq \liminf_{r\to 0^{+}}\sup_{0<\abs{t-u_0}<r}\frac{\norm{b-a}\abs{t-u_0}}{r}
                 \sup_{0<\abs{u-u_0}<r}\frac{\abs{f(T(u))-f(T(u_0))}}{\norm{T(u)-T(u_0)}} \\
           &\leq\norm{b-a}\liminf_{r\to 0^{+}}\sup_{0<\abs{u-u_0}<r}
           \frac{\abs{f(T(u))-f(T(u_0))}}{\norm{T(u)-T(u_0)}} \\
           &\leq\norm{b-a}\liminf_{r\to 0^{+}}\sup_{\norm{x-x_0}<r}
           \frac{\abs{f(x)-f(x_0)}}{\norm{x-x_0}} = \norm{b-a}\lip f(x_0). \qedhere
    \end{align*}
\end{proof}

\begin{theorem}\label{thm:littleLipLeqGamma}
	Let $D$ be a convex subset of a normed space $X$, 
    $Y$ be a metric space, $f\colon D\to Y$ be a function and $\gamma\ge 0$. 
    Then $f$ is $\gamma$-Lipschitz if and only if $\lip f(x)\le \gamma$ for any $x\in D$.
\end{theorem}
\begin{proof}
    Fix $a,b\in D$. Define
    $T\colon\cinv{0}{1}\to D$ as
    \[
        T(u)=a+u(b-a) \; \text{ for } u\in\cinv{0}{1}
    \]
    and put $g=f\circ T\colon\cinv{0}{1}\to\bR$.
    Applying Lemma \ref{lem:lip_linear_composition}
    we get
    \begin{equation*}
        \lip g(u) \le \norm{b-a}\lip f\big(T(u)\big) \leq \norm{b-a}\gamma.
    \end{equation*}
    for any $u\in\cinv{0}{1}$. Therefore, Corollary \ref{cor:BHMV} implies that 
    $g$ is Lipschitz with the constant $\gamma_1=\norm{b-a}\gamma$.
    Thus,
    \begin{align*}
        \abs{f(a)-f(b)} &= \abs{g(0)-g(1)} \\
        &\le\gamma_1\abs{0-1}\\
        &= \gamma\norm{a-b}.
    \end{align*}
    So, $f$ is $\gamma$-Lipschitz on $D$.  
\end{proof}
    By $\norm{\cdot}_{\infty}$ we denote standard norm on
    space of bounded functions $B(D,Y)$ acting
    from a set $D$ into a normed space $Y$, i.e. for any $h\colon D\to Y$
     we have 
    $\norm{h}_{\infty}=\sup\limits_{x\in D}\norm{h(x)}$.
\begin{corollary}
    Let $f\colon D\to Y$ be a continuous function, where
    $D$ is a convex subset of some normed space $X$ and $Y$ be also a
    normed space. Then,
    $\norm{f}_{\mathrm{lip}}=\norm{\lip f}_{\infty}$.
%   \[
%       \norm{f}_{\mathrm{lip}}=\sup_{x\in D}\lip f(x)=\norm{\lip f}_{\infty}.
%   \]
\end{corollary}
\begin{proof}
    We simply check, that
    \begin{align*}
        \lipnorm{f} &= \inf\set{\gamma>0}{f\text{ is }\gamma\text{-Lipschitz}} \\
        &= \inf\set{\gamma>0}{\lip f(x)\leq\gamma, \, x\in D} \\
        &= \sup\set{\lip f(x)}{x\in D} = \sup_{x\in D}\abs{\lip f(x)} = \norm{\lip f}_{\infty}, %\qedhere
    \end{align*}
    where the second equality follows from Theorem~\ref{thm:littleLipLeqGamma}
\end{proof}
Therefore,
$\lip$ is an isometric injection
of the normed space $\Lip_a(D,Y)$ 
with some $a\in X$, 
into the space $B(D,\bR)$. 

\section{Baire limit functions of Lipschitz derivatives}

%\paragraph{Baire functions.}
For a given function $f\colon X\to\overline{\bR}$, defined on a metric
space $X$, their \emph{upper Baire function} $f^{\vee}$
is defined by
\[
    f^{\vee}(x)=\inf_{U\in\Neig{x}}\sup_{u\in U}f(u), \; x\in X,
\]
and their \emph{lower Baire function} $f^{\wedge}$ is defined by
\[
    f^{\wedge}(x)=\sup_{U\in\Neig{x}}\inf_{u\in U}f(u), \; x\in X,
\]
where $\Neig{x}$ is the family of all the neighborhoods of $x$ in $X$.
(See, for example, \cite{MVONVV}.)
The upper Baire function $f^{\vee}$ is upper semicontinuous and
the lower Baire function $f^{\wedge}$ is lower semicontinuous.

A subset $D$ of a normed space $X$ is called \emph{locally convex} if for any point $x\in D$ and 
their neighborhood $U$ of $x$ in $D$ there is a convex neighborhood $V$ of $x$ in $D$ such that 
$V\subseteq U$. For example, every convex set and every open set in $X$ is locally convex. 
\begin{theorem}
    Let $D$ be a locally convex subset of a normed space $X$
    and let ${f\colon D\to\bR}$ be a function. Then
    \[
        (\lip f)^{\vee}=(\Lip f)^\vee=\LLip f.
    \]
\end{theorem}
\begin{proof}
    Since $\lip f\leq\Lip f\le \LLip f$ and $\LLip f$ is upper semicontinuous by Theorem~\ref{thm:ILip_upperSC}, we have
    \[
        (\lip f)^{\vee}\le(\Lip f)^\vee\leq(\LLip f)^{\vee}=\LLip f.
    \]
    Therefore, it is enough to prove that $\LLip f\le(\lip f)^\vee$.
    Fix $x_0\in D$. The case where  $(\lip f)^{\vee}(x_0)=\infty$ is obvious.  So, we suppose that $(\lip f)^{\vee}(x_0)<\infty$. Let $\gamma>(\lip f)^{\vee}(x_0)$.
    %Since $(\lip f)^{\vee}$ is an upper semicontinuous function, 
    Then,
    there exists a convex neighborhood $U$ of $x_0$, such that
    \[
        \lip f(x)<\gamma \, \text{ for any } x\in U.
    \]
    By Theorem~\ref{thm:littleLipLeqGamma}, the function $f$ is $\gamma$-Lipschitz on $U$.
    Hence,
    \begin{align*}
        \LLip f(x_0) &= \inf_{r>0}\lipnorm{f\big\vert_{\ball{x_0}{r}}} \\
        &\leq \lipnorm{f\big\vert_{U}} \leq \gamma,
    \end{align*}
    where $\ball{x_0}{r}$ means the ball in the metric subspace $D$.
    Passing to the limit with 
    $\gamma\to(\lip f)^{\vee}(x_0)$
    we obtain the desired inequality.
\end{proof}

%\begin{theorem}
%	Let $D$ be an open subset of a normed space $X$ and $f\colon D\to\mathbb{R}$. Then $\LLip f=(\Lip f)^{\vee}=(\lip f)^{\vee}$.
%\end{theorem}
%\begin{proof}
%    
%\end{proof}

 %  \section{Lipschitz derivatives of differentiable functions}
 %  \begin{theorem}
 %  	Let $X$ and $Y$ ne normed spaces, $D\subseteq X$, $x_0\in D\cap \mathrm{int}\overline{D}$ and $f\colon D\to Y$ be a function which is Fréchet (Gateaux ???) differentiable at $x_0$. Then $\Lip f(x_0)=\lip f(x_0)=\|Df(x_0)\|$. If moreover $f$ is Gateaux differentiable on some neighborhood of $x_0$ then $\LLip f(x_0)=\|Df\|^\vee(x_0)$. 
 %  \end{theorem}

%\section*{Acknowledgments}
%The authors would like to appreciate the referee for his/her many helpful comments and suggestions throughout this paper.




\end{document}